\documentclass[a4paper,11pt]{article}
\usepackage{graphicx}
\usepackage[mathscr]{euscript}
\usepackage[english]{babel}
\usepackage{amsfonts}
\usepackage{amsmath}
\usepackage{amsthm}
\usepackage{amssymb}
\usepackage{epsfig}
\usepackage{enumerate}
\usepackage{float}
\usepackage{mathrsfs}
\makeatletter
\newcommand{\singlespacing}{\let\CS=\@currsize\renewcommand{\baselinestretch}{1}\tiny\CS}
\newcommand{\oneandahalfspacing}{\let\CS=\@currsize\renewcommand{\baselinestretch}{1.25}\tiny\CS}
\newcommand{\doublespacing}{\let\CS=\@currsize\renewcommand{\baselinestretch}{1.5}\tiny\CS}

\newtheorem{theorem}{Theorem}[section]

\newtheorem{definition}{Definition}[section]

\newtheorem{rule-def}[theorem]{Rule}
\newtheorem{example}{Example}[section]
\newtheorem{remark}{Remark}[section]

\newcommand{\bee}{\begin{eqnarray*}}
\newcommand{\eee}{\end{eqnarray*}}

\begin{document}
\title{Modus Ponens And Modus Tollens In Linguistic Truth-Valued Propositional Logic}
 \author{Bithi Chattaraj and Sumita Basu}
\newcommand{\Addresses}{{
  \bigskip
  \footnotesize

Bithi Chattaraj(Corresponding author), \textsc{Bethune College
181, Bidhan Sarani; Kolkata-700006.} \par \nopagebreak
Email: \texttt{bithichattaraj@gmail.com}

\medskip

Sumita Basu, \textsc{Bethune College;
181 Bidhan Sarani; Kolkata-700006} \par \nopagebreak
 Email: \texttt{sumi\underline{ }basu05@yahoo.co.in}

}}

\maketitle
\begin{abstract}
Truth values of Modus Ponens and Modus Tollens rules for propositions having linguistic truth value that may be represented by lattice(Fig1, Fig2) are computed in this paper. The results show that the truth values are not always absolutely true but have graded truth values in some cases.
\end{abstract}
\section{Introduction}
Natural language is not crisp but vague. Reasoning with a crisp language is dealt in classical logic where the truth value set($C$) is $\{ T,F \}$. Logicians faced difficulty to draw inference from the sentences of natural language with two valued logic. In 1920 Lukasiewicz \cite{L21}proposed the theory of three valued logic which was later generalized to multi-valued logic. L.A Zadeh in 1975 introduced linguistic variables \cite{Z1,Z2, Z3} to capture such vague concepts. As an example \textit{Age} is a linguistic variable which may have truth values \textit{very young, moderately young, moderately old, very old}. Zadeh represented the truth values of the linguistic variable by a fuzzy set and used fuzzy logic for reasoning with linguistic variables.  Fuzzy reasoning can be viewed as a fuzzy extension of multi-valued logic.\\
Construction of suitable fuzzy set for a typical linguistic variable is very difficult. This makes problem of reasoning with linguistic variable (fuzzy reasoning) all the more challenging.  Nguyen and  Wechler \cite{N6, N7} tried to give an algebraic structure to the linguistic truth values and applied the results to fuzzy logic. Some modifications of representation of linguistic variables and its application to fuzzy reasoning  have been suggested by Di Lascio et.al  \cite{L8}, Nguyen and  Huynh, \cite{N5}, V.N. Huynh \cite {H10}, M.E Cock and E.E. Kerre \cite{M19}.\\
 The truth values of propositions of languages in real world are not exactly defined but are tagged with linguistic hedges. So given a proposition $P$ instead of saying that  \textquoteleft \textit{the proposition is true}\textquoteright,   we very often say \textit{$P$ is absolutely true/ highly true/quite true/somewhat true/rather true/slightly true etc.} Similarly, \textquoteleft \textit{the proposition is false}\textquoteright,   is replaced by the  \textit{$P$ is absolutely false/ highly false/quite false/somewhat false/rather false/slightly false etc.} The linguistic hedge set ($H$) will be $\{absolutely, highly, quite, somewhat, rather, slightly \}$. So the set truth values ($V$) of propositions of natural language will be $V= H \times C$.\\  A mapping of elements of a set $A$ to $[0, 1] $ implies that there is a linear ordering of the elements of $A$. However, in real world the elements may be incomparable. So fuzzy set theory is not adequate to deal with such non-comparable informations. In fact Zadeh(1965) commented:\textit{ \textquoteleft In a more general setting, the range of the membership function can be taken to be a suitable partially ordered set P \textquoteright}.\\  A lattice consists of a set of elements which may be comparable or non comparable.\\
In this paper we make the following assumptions:
 \begin{enumerate}
\item The linguistic hedge set ($H$) is finite and totally ordered.
\item The linguistic truth valued set $V$  forms a lattice having the Hasse diagram given either by Fig 1 or Fig 2.
\end{enumerate}
If $V$ forms a lattice of the form Fig 1 then a unary operation (inverse operation)$  "\prime "$  and a  binary operation (implication operation)  $  " \rightarrow "$ on $V$ may be defined so that  $L=(V, \vee , \wedge ,\prime ,O,I,\rightarrow )$   is a lattice implication algebra. However, if $V$ is of the form Fig 2 then $L$  is a quasi lattice implication algebra.\\
A special class of multi-valued logic called lattice-valued logic have been discussed by Y.Xu et. al \cite{ Y16, Y20}. We have used such lattice-valued logic to compute truth values of Modus Ponens and Modus Tollens rules for propositions having the linguistic truth value $V$. The result showed that the truth values are not always absolutely true.  We could also show that Modus Ponens/ Modus Tollens will have graded truth values.\\
The paper is organized as follows: In section 2, the basic properties of Quasi Lattice Implication algebra(QLIA), Lattice Implication algebra(LIA), Linguistic Truth valued propositions (LTVP) , Quasi Linguistic Truth valued propositions (QLTVP)are discussed briefly. Truth values of Modus Ponens and Modus Tollens rules for LTVP are computed in section 3 while the same for QLTVP are computed in section 4. Some concluding remarks are included in section 5.
\section{Basic Concepts}
\begin{definition}

 Let $ <L,\wedge,\vee,\prime,O,I>   $ be a bounded lattice with universal boundaries $O$ (the least element) and $I$ (the greatest element) respectively, and $"\prime"$ be an order-reversing involution. For any $ x,y,z \in L $, if mapping $\rightarrow : L \times L\rightarrow L$ satisfies:
 \begin{enumerate}
\item $(I_{1})\ x \rightarrow (y \rightarrow z)=y \rightarrow (x \rightarrow z)$
\item $(I_{2})\ x \rightarrow x=I$
\item $(I_{3})\ x \rightarrow y=y ^{\prime} \rightarrow x ^{\prime}$
\item $(I_{4})\ x \rightarrow y=y \rightarrow x=I, then x=y$
\item $(I_{5})\ ( x \rightarrow y ) \rightarrow y=( y \rightarrow x ) \rightarrow x  $
\end{enumerate}
then $ <L,\wedge,\vee,\prime,O,I,\rightarrow >  $ is called a quasi-lattice implication algebra.
If it satisfies two additional properties as follows:
 \begin{enumerate}
\item $(I_{6})\ (x \vee y ) \rightarrow z=(x \rightarrow z ) \wedge (y \rightarrow z )$  
\item $(I_{7})\ (x \wedge y ) \rightarrow z=(x \rightarrow z ) \vee  (y \rightarrow z )$
\end{enumerate}
then $ <L,\wedge,\vee,\prime,O,I,\rightarrow >  $ is called a lattice implication algebra.
\end{definition}

In classical logic the truth value of a proposition is either true or false. However, in natural language the truth values of statements are not restricted to only true or false, rather they are accompanied by some linguistic hedges which reflect the degrees of  truth or falsity of statements. Very often we  refer to a statement as {\it somewhat true} or {\it slightly false}."somewhat", "slightly" are linguistic hedges.\\Let  $H= \lbrace h_0,h_1,….,h_n | n \geq 0 \rbrace$, be the linguistic hedge set, where $h_0=slightly,\ h_1=somewhat,\ h_2=rather,\ h_n=absolutely$.The linguistic hedge operator set $H$ is totally ordered.
\begin{theorem}
Let the hedge operator set $H=\lbrace h_0,h_1,….,h_n |n \geq 0 \rbrace$, be a chain such that for $j \leq k,h_j \leq h_k, h_j \vee h_k = h_{max(j,k)}$  and $h_j \wedge h_k = h_{min(j,k)}$.
Now, we define the unary operator $"\prime"$ as ${h_j}^{\prime} = h_{n-j}$ and the binary operator $" \rightarrow "$ as $h_j \rightarrow h_k = h_{min(n,n-j+k)}$. Then the set $<H, \wedge , \vee , \prime , \rightarrow >$ is a  LIA.
\end{theorem}
\begin{proof} 
For the property $I_1:x \rightarrow (y \rightarrow z)=y \rightarrow (x \rightarrow z)$ \\
Let $x=h_i,y= h_j,z=h_k$\\
LHS = $h_{min(n,n-i+min(n,n-j+k))}$\\ 
RHS = $h_{min(n,n-j+min(n,n-i+k))}$\\ 
\begin{enumerate}
\item[i)]	$j \leq k$\\
For $i=k,i < k,i > k$\\
LHS = $h_n$ = RHS\\
\item[ii)]	$j > k$\\
For $i=k,i < k$\\
LHS = $h_n$ = RHS\\
For $i > k$\\
LHS = $h_{(min(n,2n-i-j+k))}= RHS$.\\
\end{enumerate}
For the property $I_2:x \rightarrow x=I$, our upper bound i.e. I is $h_n$ and from the definitions of implications given above we can say that the property $I_2$ is satisfied for all $h_i \in H,i\in\{0,1,…,n\}$.\\
For the property $I_3:x \rightarrow y = y^{\prime} \rightarrow x^{ \prime}$, 
Let $x=h_i,y= h_j$
LHS= $h_{(min(n,n-i+j))}$= RHS \\
For the property $I_4:x \rightarrow y=y \rightarrow x=I$, then $x=y$,
From the definition of implication given above, we can easily prove this.\\
For the property $I_5:(x \rightarrow y) \rightarrow y=(y \rightarrow x) \rightarrow x$,
	Let $x=h_i,y= h_j$\\
LHS= $h_{min(n,n-min(n,n-i+j)+j)}$ \\
RHS = $h_{min(n,n-min(n,n-j+i)+i)}$\\
\begin{enumerate}
\item[i)]If $i=j$;
LHS= $h_j,RHS=h_i$
\item[ii)]	If $i>j$;
LHS= $h_i$ =RHS
\item[iii)]	If $i<j$;
LHS= $h_j$ =RHS
\end{enumerate}
For the property $I_6:(x \vee y) \rightarrow z=(x \rightarrow z) \wedge (y \rightarrow z)$\\
Let $x=h_i,y= h_j,z=h_k$\\
LHS= $h_{min(n,n-max(i,j)+k)}$ \\
RHS = $h_{min(min(n,n-i+k),min(n,n-j+k) )}$\\ 
\begin{enumerate}
\item[i) ] $i>j$;
LHS=$h_{(min(n,n-i+k))}= RHS$
\item[ii)]$ i<j$;
LHS= $h_{(min(n,n-j+k))}$=RHS
\end{enumerate}
For the property $I_7:(x \wedge y) \rightarrow z=(x \rightarrow z) \vee (y \rightarrow z)$\\
Let $x=h_i,y= h_j,z=h_k$\\
LHS= $h_{min(n,n-min(i,j)+k)}$ \\
RHS = $h_{max(min(n,n-i+k),min(n,n-j+k) )}$ \\
\begin{enumerate}
\item[i)] $i>j$;
LHS= $h_{(min(n,n-j+k))}$= RHS
\item[ii)]$ i<j$;
LHS= $h_{(min(n,n-i+k))}$= RHS
\end{enumerate}
Thus, from the above properties we can see that the set $<H, \wedge , \vee ,\prime , \rightarrow >$ forms a LIA.
\end{proof}
	The basic truth value set is $C=\lbrace T,F \rbrace,$ where $T=true,\ F=false$. \\Let $V$ be the set of all linguistic truth values, i.e. $V=H \times C$. Thus, if $v$ is a linguistic truth value then, $, v\in V$ and $v=(h_i,c_j)$ where $h_i\in H,c_j\in C$ is composed of a linguistic hedge operator $h_i$ and a basic truth value $c_j$. 
If $V=\lbrace v_{00},v_{01},v_{10},v_{11},...,v_{n0},v_{n1} \rbrace$ then $v_{i0}=(h_i,c_0)=(h_i,F)$ and $ v_{i1}=(h_i,c_1)=(h_i,T)$. So if $h_i$ represents $"somewhat"$ then $v_{i0}$ represents $"somewhat$ $false"$ and $v_{i1}$ represents $"somewhat$ $true".$  \\
Let, $V_1=\{v_{i1}|i=0,1,2,...n\}$ and $V_0=\{v_{i0}|i=0,1,2,...n\}$ so that $V = V_0 \cup V_1$.Also $V_0, V_1$ satisfy the following:
  $$ v_{i1}, v_{j1} \in V_1 ,i \leq j \Rightarrow v_{i1} \leq v_{j1}$$ and,$$ v_{i0}, v_{j0} \in V_0 ,i \leq j \Rightarrow v_{j1} \leq v_{i1}$$
So,both  $V_0$ and $V_1$ are totally ordered and may be represented by chain.\\
Also the elements of $V$ have the following order
$$ k \in\{0,1,...n\}, v_{k0} \leq v_{(n-k)1}$$
Hence $V$ is a partially ordered set and may be represented by the following Hasse diagram.
\begin{figure}[h]
\centering
\begin{minipage}[b]{2in}
\hfill
  \includegraphics[width=2in, height=3.0in]{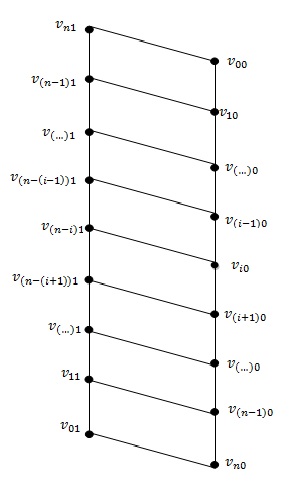}
\caption{Hasse Diagram of Linguistic Truth Value set V}
\label{fig:fig2}
\end{minipage}
 \end{figure}
\begin{theorem}
Let $V$ be the linguistic truth value set, then $(V,\leq)$ is a poset.  If $\mathscr{V}= (V,\vee , \wedge, O,I)$ where $O= v_{n0}, I= v_{n1}$ and the operation $"\vee "$ and $"\wedge "$ are defined as 
$\forall i, j \in \{0,1,..n\}, i\leq j $
\begin{enumerate}
\item $v_{i1} \vee v_{j1} = v_{j1}$
\item$ v_{i0} \vee v_{j0} = v_{i0}$
\item $v_{i1} \vee v_{j0} = v_ {i1},~ if~ n\leq (i+j)$
\item $v_{i1} \vee v_{j0} = v_ {(n-j)1},~ if ~n\geq (i+j)$
\item $v_{i1} \wedge v_{j1} = v_{i1}$
\item $v_{i0} \wedge v_{j0} = v_{j0}$
\item $v_{i1} \wedge v_{j0} = v_ {j0},~ if~ n\leq (i+j)$
\item $v_{i1} \wedge v_{j0} = v_ {(n-i)0},~ if ~n\geq (i+j)$
\end{enumerate}
then $\mathscr{V}$ forms a lattice.
\end{theorem}
\begin{proof}
In the properties $1$, $2$, $5$ and $6$ if we put $i=j$ then we see that the idempotent law is satisfied.\\
From the above definitions of $\vee$ and $\wedge$ we see that the commutative law is also satisfied.\\
Now we see whether associative law holds.\\
\begin{enumerate}
\item[a)]Let $x=v_{i1}, y= v_{j1}, z= v_{k1}, \forall i,j,k \in \{0,1,2,..n\}$\\
$v_{i1}\wedge (v_{j1}\wedge v_{k1})= v_{i1}= (v_{i1}\wedge v_{j1}) \wedge v_{k1}$ if $i\leq j\leq k$.\\
$v_{i1}\vee (v_{j1}\vee v_{k1})= v_{k1}= (v_{i1}\vee v_{j1}) \vee v_{k1}$ if $i\leq j\leq k$.\\

\item[b)]Let $x=v_{i1}, y= v_{j1}, z= v_{k0}, \forall i,j,k \in \{0,1,2,..n\}$\\
$
v_{i1}\wedge (v_{j1}\wedge v_{k0})=
\begin{cases}
v_{k0} &\text{if}\  n\leq (i+k)\\ 
v_{(n-i)0} &\text{if} \ n\geq (i+k)\\
\end{cases}
$
$=(v_{i1}\wedge v_{j1}) \wedge v_{k0}$ if $i\leq j\leq k$\\
$
v_{i1}\vee (v_{j1}\vee v_{k0})=
\begin{cases}
v_{j1} &\text{if}\  n\leq (j+k)\\ 
v_{(n-k)1} &\text{if} \ n\geq (j+k)\\
\end{cases}
$
$=(v_{i1}\vee v_{j1}) \vee v_{k0}$ if $i\leq j\leq k$

\item[c)]Let $x=v_{i1}, y= v_{j0}, z= v_{k1}, \forall i,j,k \in \{0,1,2,..n\}$\\
$
v_{i1}\wedge (v_{j0}\wedge v_{k1})=
\begin{cases}
v_{j0} &\text{if}\  n\leq (i+j)\\ 
v_{(n-i)0} &\text{if} \ n\geq (i+j)\\
\end{cases}
$
$=(v_{i1}\wedge v_{j0}) \wedge v_{k1}$ if $i\leq j\leq k$\\
$
v_{i1}\vee (v_{j0}\vee v_{k1})=
\begin{cases}
v_{k1} &\text{if}\  n\geq (i+j); \ n\leq (j+k)\\ 
v_{(n-j)1} &\text{if} \ n\geq (j+k); \ n\geq (j+k)\\
\end{cases}
$
$=(v_{i1}\vee v_{j0}) \vee v_{k1}$ if $i\leq j\leq k$

\item[d)]Let $x=v_{i1}, y= v_{j0}, z= v_{k0}, \forall i,j,k \in \{0,1,2,..n\}$\\
$
v_{i1}\wedge (v_{j0}\wedge v_{k0})=
\begin{cases}
v_{k0} &\text{if}\  n\leq (i+j)\\ 
v_{(n-i)0} &\text{if} \ n\geq (i+k)\\
\end{cases}
$
$=(v_{i1}\wedge v_{j0}) \wedge v_{k0}$ if $i\leq j\leq k$\\
$
v_{i1}\vee (v_{j0}\vee v_{k0})=
\begin{cases}
v_{i1} &\text{if}\  n\leq (i+j)\\ 
v_{(n-j)1} &\text{if} \ n\geq (i+j)\\
\end{cases}
$
$=(v_{i1}\vee v_{j0}) \vee v_{k0}$ if $i\leq j\leq k$

\item[e)]Let $x=v_{i0}, y= v_{j1}, z= v_{k1}, \forall i,j,k \in \{0,1,2,..n\}$\\
$
v_{i0}\wedge (v_{j1}\wedge v_{k1})=
\begin{cases}
v_{i0} &\text{if}\  n\leq (i+j)\\ 
v_{(n-j)0} &\text{if} \ n\geq (i+j)\\
\end{cases}
$
$=(v_{i0}\wedge v_{j1}) \wedge v_{k1}$ if $i\leq j\leq k$\\
$
v_{i0}\vee (v_{j1}\vee v_{k1})=
\begin{cases}
v_{k1} &\text{if}\  n\leq (i+j)\\ 
v_{(n-i)1} &\text{if} \ n\geq (i+k)\\
\end{cases}
$
$=(v_{i0}\vee v_{j1}) \vee v_{k1}$ if $i\leq j\leq k$

\item[f)]Let $x=v_{i0}, y= v_{j1}, z= v_{k0}, \forall i,j,k \in \{0,1,2,..n\}$\\
$
v_{i0}\wedge (v_{j1}\wedge v_{k0})=
\begin{cases}
v_{k0} &\text{if}\ n\geq (i+j);\  n\leq (j+k)\\ 
v_{(n-j)0} &\text{if} \ n\geq (i+j);\ n\geq (j+k)\\
\end{cases}
$
$=(v_{i0}\wedge v_{j1}) \wedge v_{k0}$ if $i\leq j\leq k$\\
$
v_{i0}\vee (v_{j1}\vee v_{k0})=
\begin{cases}
v_{j1} &\text{if}\  n\leq (i+j);\ n\leq (j+k)\\ 
v_{(n-i)1} &\text{if} \ n\geq (i+j);\ n\leq (j+k)\\
\end{cases}
$
$=(v_{i0}\vee v_{j1}) \vee v_{k0}$ if $i\leq j\leq k$

\item[g)]Let $x=v_{i0}, y= v_{j0}, z= v_{k1}, \forall i,j,k \in \{0,1,2,..n\}$\\
$
v_{i0}\wedge (v_{j0}\wedge v_{k1})=
\begin{cases}
v_{j0} &\text{if}\ n\leq (j+k)\\ 
v_{(n-k)0} &\text{if} \ n\geq (i+k);\ n\geq (j+k)\\
\end{cases}
$
$=(v_{i0}\wedge v_{j0}) \wedge v_{k1}$ if $i\leq j\leq k$\\
$
v_{i0}\vee (v_{j0}\vee v_{k1})=
\begin{cases}
v_{k1} &\text{if}\  n\leq (i+k);\ n\leq (j+k)\\ 
v_{(n-i)1} &\text{if} \ n\geq (i+k);\ n\leq (j+k)\\
\end{cases}
$
$=(v_{i0}\vee v_{j0}) \vee v_{k1}$ if $i\leq j\leq k$

\item[h)]Let $x=v_{i0}, y= v_{j0}, z= v_{k0}, \forall i,j,k \in \{0,1,2,..n\}$\\
$
v_{i0}\wedge (v_{j0}\wedge v_{k0})= v_{k0} =(v_{i0}\wedge v_{j0}) \wedge v_{k0}$ if $i\leq j\leq k$\\
$
v_{i0}\vee (v_{j0}\vee v_{k0})= v_{i0} =(v_{i0}\vee v_{j0}) \vee v_{k0}$ if $i\leq j\leq k$
\end{enumerate}
Therefore, we see that the associative law is also satisfied.\\
Now, we check for the absorption law.

\begin{enumerate}
\item[a)]Let $x=v_{i1}, y= v_{j1}, \forall i,j \in \{0,1,2,..n\}$\\
$v_{i1} \wedge (v_{i1} \vee v_{j1})= v_{i1}= v_{i1} \vee (v_{i1} \wedge v_{j1})$

\item[b)]Let $x=v_{i1}, y= v_{j0}, \forall i,j \in \{0,1,2,..n\}$\\
$v_{i1} \wedge (v_{i1} \vee v_{j0})= v_{i1} = v_{i1} \vee (v_{i1} \wedge v_{j0})$

\item[c)]Let $x=v_{i0}, y= v_{j1}, \forall i,j \in \{0,1,2,..n\}$\\
$v_{i0} \wedge (v_{i0} \vee v_{j1})= v_{i0} = v_{i0} \vee (v_{i0} \wedge v_{j1})$

\item[d)]Let $x=v_{i0}, y= v_{j0}, \forall i,j \in \{0,1,2,..n\}$\\
$v_{i0} \wedge (v_{i0} \vee v_{j0})= v_{i0} = v_{i0} \vee (v_{i0} \wedge v_{j0})$.
\end{enumerate}
Thus, the absorption law is also satisfied.\\
Therefore, $\mathscr{V}$ forms a lattice.
\end{proof}
\begin{definition}
Let $\mathscr{V}= (V,\vee , \wedge, O,I)$ be a lattice as defined in the previous theorem. We define $\forall i, j \in \{0,1\}$,\\ A unary operation (inverse operation)$  "\prime "$  on $V$ as $$((h_i,T)^ {\prime} =(h_i,F)) \wedge ((h_i,F)^ {\prime} =(h_i,T)) \Rightarrow  v_{ij}^{\prime}= v_{i(1-j)}$$ A binary operation (implication operation)  $  " \rightarrow "$ is also defined on $V$ as follows:
\begin{enumerate}
\item $((h_i,T) \rightarrow (h_j,F)=(h_{max(0,i+j-n)} ,F)) \Rightarrow( (v_{i1} \rightarrow v_{j0}) = v_{max(0,i+j-n)0})$
\item $((h_i,F) \rightarrow (h_j,T)=(h_{min(n,i+j)} ,T))\Rightarrow( (v_{i0} \rightarrow v_{j1}) = v_{min(n,i+j)1})$
\item $(h_i,T) \rightarrow (h_j,T)=(h_{(min(n,n-i+j))},T)\Rightarrow( (v_{i1} \rightarrow v_{j1}) = v_{min(n,n-i+j)1})$
\item$(h_i,F) \rightarrow (h_j,F)=(h_{(min(n,n-j+i))},T)\Rightarrow( (v_{i0} \rightarrow v_{j0}) = v_{min(n, n-j+i)1})$
\end{enumerate}
\end{definition}
\begin{theorem}
If $L=(V, \vee , \wedge ,\prime ,O,I,\rightarrow )$ then $L$  is a lattice implication algebra.
\end{theorem}
\begin{proof} 
$(V, \vee , \wedge ,\prime ,O,I,\rightarrow )$ forms a bounded lattice due to Theorem 2.2.\\$"\prime"$ is an order reversing involution as $$ \forall i\in \{1,2,..n\}, \forall j \in \{0, 1\},  ((v_{ij})^\prime)^\prime = (v_{i(1-j)})^\prime=v_{ij}$$
For the property $I_1:x \rightarrow (y \rightarrow z)=y \rightarrow (x \rightarrow z)$\\
\begin{enumerate}
\item[a)]	Let $x=v_{i1}, y= v_{j1}, z= v_{k1}, \forall i,j,k \in \{0,1,2,..n\}$\\
LHS = $v_{min(n,2n-i-j+k)1}$=RHS if $i,j >k$\\
LHS=$v_{n1}$= RHS if $i,j \leq k$; $j \leq k ,\ i>k$ and $j>k ,\ i \leq k$.

\item[b)]	Let $x=v_{i1}, y= v_{j1}, z= v_{k0}, \forall i,j,k \in \{0,1,2,..n\}$\\
LHS= $v_{(max(0,i+j+k-2n))0}$= RHS if $i,j>(n-k)$\\
LHS= $v_{00}$= RHS if $i,j \leq (n-k)$; $j \leq (n-k),\ i> (n-k)$ and $j> (n-k),\ i \leq (n-k)$.

\item[c)]	Let $x=v_{i1}, y= v_{j0}, z= v_{k1}, \forall i,j,k \in \{0,1,2,..n\}$\\
LHS= $v_{(min(n,n-i+j+k))1}$= RHS if $j \leq (n-k),\ i>k$\\
LHS= $v_{n1}$= RHS if $i,j \geq (n-k)$; $j \geq (n-k),\ i< k$ and $j<(n-k),\ i \leq k$.

\item[d)] Let $x=v_{i1}, y= v_{j0}, z= v_{k0}, \forall i,j,k \in \{0,1,2,..n\}$\\
LHS= $v_{(min(n,2n-i+j-k))1}$ = RHS if $i> (n-k), j<k$\\
LHS= $v_{n1}$ if $i \leq (n-k),\ j \geq k$; $i> (n-k), j \geq k$ and $i \leq (n-k),\ j<k$.

\item[e)]Let $x=v_{i0}, y= v_{j1}, z= v_{k1}, \forall i,j,k \in \{0,1,2,..n\}$\\
LHS= $v_{(min(n,n+i-j+k))1}$ = RHS if $i< (n-k),j>k$\\
LHS= $v_{n1}$=RHS if $i \geq (n-k),j \leq k$; $i<(n-k),\ j \leq k$ and $i \geq (n-k),\ j>k$.

\item[f)]Let $x=v_{i0}, y= v_{j1}, z= v_{k0}, \forall i,j,k \in \{0,1,2,..n\}$\\
LHS =$v_{(min(n,2n+i-j-k))1}$=RHS if $i<k,j>(n-k)$\\
LHS= $v_{n1}$ if $i \geq k,j \leq (n-k)$; $i<k,\ j \leq (n-k)$ and $i \geq k,\ j>(n-k)$.

\item[g)]Let $x=v_{i0}, y= v_{j0}, z= v_{k1}, \forall i,j,k \in \{0,1,2,..n\}$\\
LHS = $v_{(min(n,i+j+k))1}$= RHS if $i,j<(n-k)$\\
LHS = $v_{n1}$= RHS if $i,j \geq (n-k)$; $i<(n-k),\ j \geq (n-k)$ and $i \geq (n-k),\ j<(n-k)$.

\item[h)]Let $x=v_{i0}, y= v_{j0}, z= v_{k0}, \forall i,j,k \in \{0,1,2,..n\}$\\
LHS = $v_{(min(n,n+i+j-k))1}$= RHS if $i,j<k$\\
LHS = $v_{n1}$ if $i,j \geq k$; $i<k,\ j \geq k$ and $i \geq k,\ j<k$.
\end{enumerate}

For the property $I_2:x \rightarrow x = I$, either $x=v_{i0}~ or ~v_{j1}$. In either case  $x \rightarrow x = v_{n1}= I$.\\\\
Property $I_3$ is satisfied as we have the following:
$$x=v_{i0}, y=v_{j0} \Rightarrow (x\rightarrow y = v_{min(n,n-j+i)1}=y^\prime\rightarrow x^\prime)$$
$$x=v_{i0}, y=v_{j1} \Rightarrow (x\rightarrow y = v_{min(n,j+i)1}=y^\prime\rightarrow x^\prime)$$
$$x=v_{i1}, y=v_{j1} \Rightarrow (x\rightarrow y = v_{min(n,n-i+j)1}=y^\prime\rightarrow x^\prime)$$
$$x=v_{i1}, y=v_{j0} \Rightarrow (x\rightarrow y = v_{max(0,j+i-n)0}=y^\prime\rightarrow x^\prime)$$
Property $I_4$ is satisfied as we have the following:
$$x=v_{i0}, y=v_{j0},((v_{i0}\rightarrow v_{j0})=(v_{j0}\rightarrow v_{i0})=v_{n1} )\Rightarrow ( v_{min(n,n-j+i)1} =v_{min(n,n-i+j)1}= v_n) $$
$$\Rightarrow( (n \leq (n-i+j))\wedge (n \leq (n-j+i)) \Rightarrow (i=j))$$
$$x=v_{i1}, y=v_{j1},( (x\rightarrow y = v_{min(n,n-i+j)1}=(y\rightarrow x)=  v_{min(n,n-i+j)1}=v_{n1})\Rightarrow (i=j)$$
If $x=v_{i1}, y=v_{j0}$ then $ (x\rightarrow y = v_{max(0,j+i-n)0} \neq v_{n1})$, hence the required conditions are not satisfied.\\\\
Property $I_5$,\\
Case1\\
If, $x=v_{i1}, y=v_{j1}$ then, $$((x\rightarrow y)\rightarrow y)=v_{min(n,n-min(n,n-i+j)+j)1}~ and ~((y\rightarrow x)\rightarrow x)=v_{min(n,n-min(n,n-j+i)+i)1}$$So, for $$j \leq i,~((x\rightarrow y)\rightarrow y)=v_{i1} = ((y\rightarrow x)\rightarrow x)$$ and for $$i < j,~((x\rightarrow y)\rightarrow y)=v_{j1} = ((y\rightarrow x)\rightarrow x)$$ 
Case 2\\
Similarly, if $x=v_{i0}, y=v_{j0}$, $$j \leq i,~((x\rightarrow y)\rightarrow y)=v_{j0} = ((y\rightarrow x)\rightarrow x)$$ and for $$i < j,~((x\rightarrow y)\rightarrow y)=v_{i0} = ((y\rightarrow x)\rightarrow x)$$ 
Case 3\\
If, $x=v_{i1}, y=v_{j0}$,\\If $ n \leq (i+j)$,$$(v_{i1}\rightarrow v_{j0})\rightarrow v_{j0} = v_{i1} = (v_{j0}\rightarrow v_{i1})\rightarrow v_{i1}$$
If $ (i+j) \leq n$,$$(v_{i1}\rightarrow v_{j0})\rightarrow v_{j0} = v_{(n-j)1} = (v_{j0}\rightarrow v_{i1})\rightarrow v_{i1}$$
Case 4\\
If, $x=v_{i0}, y=v_{j1}$,\\If $ n \leq (i+j)$,$$(v_{i0}\rightarrow v_{j1})\rightarrow v_{j1} = v_{j1} = (v_{j1}\rightarrow v_{i0})\rightarrow v_{i0}$$
If $ (i+j) \leq n$,$$(v_{i0}\rightarrow v_{j1})\rightarrow v_{j1} = v_{(n-i)1} = (v_{j1}\rightarrow v_{i0})\rightarrow v_{i0}$$.
Now, let us check the properties $I_6: (x\vee y)\rightarrow z= (x\rightarrow z)\wedge (y\rightarrow z)$ and $I_7: (x\wedge y)\rightarrow z= (x\rightarrow z)\vee (y\rightarrow z)$.\\
\begin{enumerate}
\item[a)]	Let $x=v_{i1}, y= v_{j1}, z= v_{k1}, \forall i,j,k \in \{0,1,2,..n\}$\\
For property $I_6:$,LHS = $v_{min(n,n-i+k)1}$ = RHS, if $i \geq j$\\ 
LHS = $v_{min(n,n-j+k)1}$ = RHS, if $j>i$.\\
 
For property $I_7:$,LHS = $v_{min(n,n-j+k)1}$ = RHS, if $i \geq j$\\
LHS = $v_{min(n,n-i+k)1}$ = RHS, if $j>i$.

\item[b)]	Let $x=v_{i1}, y= v_{j1}, z= v_{k0}, \forall i,j,k \in \{0,1,2,..n\}$\\
For property $I_6$, LHS = $v_{max(0,i+k-n)0}$ = RHS, if $i \geq j$\\
LHS = $v_{max(0,j+k-n)0}$ = RHS, if $j>i$.\\

For property $I_7$, LHS = $v_{max(0,j+k-n)0}$ = RHS, if $i \geq j$\\
LHS = $v_{max(0,i+k-n)0}$ = RHS, if $j>i$.

\item[c)]	Let $x=v_{i1}, y= v_{j0}, z= v_{k1}, \forall i,j,k \in \{0,1,2,..n\}$\\
For property $I_6$, LHS = $v_{min(n,n-i+k)1}$ = RHS, if $i,j \geq \dfrac{n}{2}$\\
LHS = $v_{min(n,j+k)1}$ = RHS, if $i,j< \dfrac{n}{2}$.\\ 

For property $I_7$, LHS = $v_{min(n,j+k)1}$ = RHS, if $i,j \geq \dfrac{n}{2}$\\
LHS = $v_{min(n,n-i+k)1}$ = RHS, if $i,j< \dfrac{n}{2}$.

\item[d)]Let $x=v_{i1}, y= v_{j0}, z= v_{k0}, \forall i,j,k \in \{0,1,2,..n\}$\\
For property $I_6$, LHS = $v_{max(0,i+k-n)0}$ = RHS, if $i,j \geq \dfrac{n}{2}$\\
LHS = $v_{min(n,n-k+j)1}$ = RHS, if $i,j< \dfrac{n}{2}$.\\

For property $I_7$, LHS = $v_{min(n,n-k+j)1}$ = RHS, if $i,j \geq \dfrac{n}{2}$\\
LHS = $v_{max(0,i+k-n)0}$ = RHS, if $i,j< \dfrac{n}{2}$.

\item[e)]	Let $x=v_{i0}, y= v_{j1}, z= v_{k1}, \forall i,j,k \in \{0,1,2,..n\}$\\
For property $I_6$, LHS = $v_{min(n,n-j+k)1}$ = RHS, if $i,j \geq \dfrac{n}{2}$\\
LHS = $v_{min(n,j+k)1}$ = RHS, if $i,j< \dfrac{n}{2}$.\\ 

For property $I_7$, LHS = $v_{min(n,i+k)1}$ = RHS, if $i,j \geq \dfrac{n}{2}$\\
LHS = $v_{min(n,n-i+k)1}$ = RHS, if $i,j< \dfrac{n}{2}$. 

\item[f)]	Let $x=v_{i0}, y= v_{j1}, z= v_{k0}, \forall i,j,k \in \{0,1,2,..n\}$\\
For property $I_6$, LHS = $v_{max(0,j+k-n)0}$ = RHS, if $i,j \geq \dfrac{n}{2}$\\
LHS = $v_{min(n,n-k+i)1}$ = RHS, if $i,j< \dfrac{n}{2}$.\\ 

For property $I_7$, LHS = $v_{min(n,n-k+i)1}$ = RHS, if $i,j \geq \dfrac{n}{2}$\\
LHS = $v_{max(0,j+k-n)0}$ = RHS, if $i,j< \dfrac{n}{2}$.

\item[g)]	Let $x=v_{i0}, y= v_{j0}, z= v_{k1}, \forall i,j,k \in \{0,1,2,..n\}$\\
For property $I_6$, LHS = $v_{min(n,j+k)1}$ = RHS, if $i \geq j$\\
LHS = $v_{min(n,i+k)1}$ = RHS, if $i< j$.\\
 
For property $I_7$, LHS = $v_{min(n,i+k)1}$ = RHS, if $i \geq j$\\
LHS = $v_{min(n,j+k)1}$ = RHS, if $i< j$. 

\item[h)]Let $x=v_{i0}, y= v_{j0}, z= v_{k0}, \forall i,j,k \in \{0,1,2,..n\}$\\
For property $I_6$, LHS = $v_{min(n,n-k+j)1}$ = RHS, if $i \geq j$\\
LHS = $v_{min(n,n-k+i)1}$ = RHS, if $i< j$.\\ 

For property $I_7$, LHS = $v_{min(n,n-k+i)1}$ = RHS, if $i \geq j$\\
LHS = $v_{min(n,n-k+j)1}$ = RHS, if $i< j$. 

\end{enumerate}
Thus, $L$ is a \textit{lattice implication algebra }.
\end{proof}
\begin{definition} Propositions having linguistic truth values are called Linguistic truth-valued Propositions (LTVP) and is denoted by $\mathscr{P}$ . If $e$ is a truth valuation of a LTVP  then $e: \mathscr{P}\rightarrow L$. 
\end{definition}
\begin{definition} Let $P \in \mathscr{P}$ denote an atom, the fundamental element of LTVP. Any  formula of LTVP is defined recursively as follows:
\begin{enumerate}
\item $P$ is a formula.
\item If $G$ is a formula then $G^\prime$ is also a formula.
\item If $G, H$ are LTVP formulae then $G\wedge H, G\vee H, G \rightarrow H$ are formulae.
\item Any symbolic string formed using 1,2,3 a finite number of times is called a formulae in LTVP. No other string is a formulae.
\end{enumerate}
\end{definition}
\begin{definition}Truth value of a formula of LTVP is defined recursively as follows:
\begin{enumerate}
\item If $G \in \mathscr{P} $ then $e(G^\prime)= (e(G))^\prime$ is also a formula.
\item If $G, H \in \mathscr{P}$  then $e(G\wedge H)= e(G)\wedge e(H); e(G\vee H)= e(G)\vee e(H); e(G \rightarrow H)= e(G) \rightarrow e(H)$
\end{enumerate}
\end{definition}
Next we discuss the case where all the elements of $V$ do not satisfy the following order $ i \in\{0,1,...n\}, v_{i0} \leq v_{(n-i)1}$, i.e $\exists $ at least one $i$ for which the elements are non- comparable, 
then also $V$ is a partially ordered set and may be represented by the following (Figure 2) Hasse diagram.
\begin{figure}[h]
\centering
\begin{minipage}[b]{3in}
\hfill
  \includegraphics[width=2in, height=3.0in]{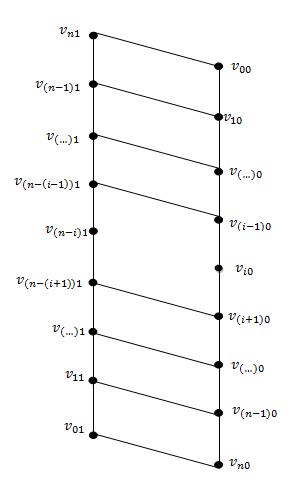}
\caption{Hasse Diagram of Linguistic Truth Value set V satisfying QLIA}
\end{minipage}
 \end{figure}
\begin{theorem}
Let $V$ be the linguistic truth value set, then $(V,\leq)$ is a poset.  If $\mathscr{V}= (V,\vee , \wedge, O,I)$ where $O= v_{n0}, I= v_{n1}$ and $v_{i0}$ and $v_{(n-i)1}$ are non comparable $(Figure 2)$, and the operation $"\vee "$ and $"\wedge "$ are defined as 
$\forall k, l \in \{0,1,..n\}, k\leq l $ 
\begin{enumerate}
\item $v_{k1} \vee v_{l1} = v_{l1}$
\item$ v_{k0} \vee v_{l0} = v_{k0}$
\item $v_{k1} \vee v_{l0} = 
\begin{cases}
v_ {k1},\ k \neq (n-i)\\
v_{(n-(i-1))1},\ k= (n-i) \end{cases} ,~ if~ n\leq (k+l)$
\item $v_{k1} \vee v_{l0} = 
\begin{cases}
v_ {(n-l)1},\ l\neq i\\
v_{(n-(i-1))1},\ l=i \end{cases}, ~ if ~n\geq (k+l)$
\item $v_{k1} \wedge v_{l1} = v_{i1}$
\item $v_{k0} \wedge v_{l0} = v_{j0}$
\item $v_{k1} \wedge v_{l0} = 
\begin{cases}
v_ {l0}\\
v_{(i+1)0},\ k=(n-i),\ l=i \end{cases}, ~ if~ n\leq (k+l)$
\item $v_{k1} \wedge v_{l0} = 
\begin{cases}
v_ {(n-k)0},\ k\neq (n-i)\\
v_{(i+1)0},\ k=(n-i)\end{cases}, ~ if ~n\geq (k+l)$
\end{enumerate}
then $\mathscr{V}$ forms a lattice.
\end{theorem}
\begin{proof}
The proof is similar to that of the theorem $2.2$.
\end{proof}
With the unary operation (inverse operation) $"\prime"$ and the binary operation (implication operation)  $"\rightarrow "$ defined on $V$ as given in the definition $2.2$ we can prove that the structure $L_1=(V,\wedge, \vee,\prime,O,I,\rightarrow)$ where $V$ is the set given above forms a Quasi lattice implication algebra.
\begin{theorem}
If $L_1=(V, \vee , \wedge ,\prime ,O,I,\rightarrow )$ then $L_1$  is a Quasi lattice implication algebra.
\end{theorem}
\begin{proof} 
$(V, \vee , \wedge ,\prime ,O,I,\rightarrow )$ forms a bounded lattice due to theorem $2.4$.\\$"\prime"$ is an order reversing involution as $$ \forall i\in \{1,2,..n\}, \forall j \in \{0, 1\},  ((v_{ij})^\prime)^\prime = (v_{i(1-j)})^\prime=v_{ij}$$
The proof is similar to the theorem $2.3$ for the properties $I_1-I_5$. Now we check for the properties $I_6: (x\vee y)\rightarrow z= (x\rightarrow z)\wedge (y\rightarrow z)$ and $I_7: (x\wedge y)\rightarrow z= (x\rightarrow z)\vee (y\rightarrow z)$.\\

Let $x=v_{(n-i)1}, y= v_{i0}, z= v_{k1}$ for any $i \in \lbrace 1,2,...,(n-1)\rbrace$ and $k \in \lbrace 0,1,2,...,n\rbrace$ with $i+k+1<n$\\
For property $I_6:$
LHS = $v_{(i+k-1)1}$ \\ RHS = $v_{(i+k)1}$.\\ 
For property $I_7:$
LHS = $v_{(i+k+1)1}$ \\ RHS = $v_{(i+k)1}$.\\
Thus, we see that $L_1$ does not satisfy the properties $I_6$ and $I_7$ and therefore it is a \textit{Quasi lattice implication algebra}.
\end{proof}
\begin{definition} Propositions having Quasi linguistic truth values are called Quasi Linguistic truth-valued Propositions (QLTVP) and is denoted by $\mathscr{Q}$ . If $e$ is a truth valuation of a QLTVP  then $e: \mathscr{Q}\rightarrow L_{1}$. 
\end{definition}
\begin{definition} Let $Q \in \mathscr{Q}$ denote an atom, the fundamental element of QLTVP. Any  formula of QLTVP is defined recursively as follows:
\begin{enumerate}
\item $Q$ is a formula.
\item If $G$ is a formula then $G^\prime$ is also a formula.
\item If $G, H$ are QLTVP formulae then $G\wedge H, G\vee H, G \rightarrow H$ are formulae.
\item Any symbolic string formed using 1,2,3 a finite number of times is called a formulae in QLTVP. No other string is a formulae.
\end{enumerate}
\end{definition}

\section{Reasoning In Linguistic Truth-valued Propositions}
 Truth values of Modus Ponens and Modus Tollens rules for Linguistic truth-valued Propositions (LTVP) are computed in this section. It is observed that the fundamental reasoning tools are not always absolutely true for LTVP.\\
Let $\mathscr{P}$ = \{P/ P is a LTVP\} be the set of all LTVPs. The \textit{truth evaluation e} is a function given by $e: \mathscr{P} \rightarrow V $. Henceforth we will denote $P^\prime $ by $\neg P$.

\begin{theorem}If $P,Q \in \mathscr{P}$ and $e(P)=v_{i1}; e(Q)= v_{j1}$,  then the truth values of the Modus Ponens and Modus Tollens are as  follows:
\begin{enumerate}
 \item $e\{(P\bigwedge(P\rightarrow Q))\rightarrow Q\}= 
 \begin{cases}
 v_{n1},&\text{ if} \ (i \leq j)\\ 
v_{(n-i+j)1} &\text{if}\ (i \geq j)~and~ 2i\leq (n+j)\\ 
v_{i1} &\text{if} \ (i \geq j)~and~ 2i\geq (n+j)\\
\end{cases}$\\
\item  $e\{(\neg Q)\bigwedge(P\rightarrow Q) \rightarrow(\neg P)\}= 
 \begin{cases}
 v_{n1},&\text{ if} \ (i \leq j)\\ 
v_{(n-i+j)1} &\text{if}\ j \leq i\leq 2j\\ 
v_{(n-j)1} &\text{if} \ i> 2j >j\\
\end{cases}$\\
\end{enumerate}
\end{theorem}
\begin{proof} Case 1a: Let $i\leq j$,\\
$ e(P\rightarrow Q)= v_{n1}$.\\
So, $e(P\wedge (P\rightarrow Q))= 
e(P)\wedge e(P\rightarrow Q)= v_{i1}$\\
 $e(P\wedge(P\rightarrow Q)\rightarrow Q)= e(P\wedge(P\rightarrow Q))\rightarrow e(Q)= v_{n1}$\\
Case1b:Let $i>j$,\\
$e(P\rightarrow Q)= v_{(n-i+j)1}$.\\
So, $e(P\wedge (P\rightarrow Q))= 
e(P)\wedge e(P\rightarrow Q)= \begin{cases}
v_{i1} &\text{if}\  2i\leq (n+j)\\ 
v_{(n-i+j)1} &\text{if} \ 2i\geq (n+j)\\
\end{cases}$\\
$e(P\wedge(P\rightarrow Q)\rightarrow Q)= e(P\wedge(P\rightarrow Q))\rightarrow e(Q)= 
\begin{cases}
v_{(n-i+j)1} &\text{if}\  2i\leq (n+j)\\ 
v_{i1} &\text{if} \ 2i\geq (n+j)\\
\end{cases}$\\
Case2a: Let $i\leq j$,\\
 $e(\neg Q\wedge (P\rightarrow Q))= 
e(\neg Q)\wedge e(P\rightarrow Q)= v_{j0}$\\
 $e(\neg Q\wedge(P\rightarrow Q)\rightarrow \neg P)= e(\neg Q\wedge(P\rightarrow Q))\rightarrow e(\neg P)= v_{n1}$.\\
Case 2b:Let $i>j$,\\
 $e(\neg Q\wedge (P\rightarrow Q))= 
e(\neg Q)\wedge e(P\rightarrow Q)= \begin{cases}
v_{j0} &\text{if}\  i\leq 2j\\ 
v_{(i-j)0} &\text{if} \ i\geq 2j\\
\end{cases}$\\
$e(\neg Q\wedge(P\rightarrow Q)\rightarrow \neg P)= e(\neg Q\wedge(P\rightarrow Q))\rightarrow e(\neg P)= \begin{cases}
v_{(n-i+j)1} &\text{if}\  i\leq 2j\\ 
v_{(n-j)1} &\text{if} \ i> 2j\\
\end{cases}$\\
\end{proof}
\begin{remark} If the truth evaluation of the propositions $P$ and $Q$ is restricted to $V_1  = \{v_{01},v_{11},v_{21}..\\..,v_{n1}\}$ and $e(P)$ is less than or equal to $e(Q)$ then both Modus Ponens and Modus Tollens rules are absolutely true.\\
If $ e(P)= e(Q)=v_{n1}$ then $P$ and $Q$ are classical propositions and Modus Ponens  and  Modus Tollens rules are absolutely true in such cases.\\
However, if  $e(P)$ is greater than or equal to $e(Q)$ then both Modus Ponens and Modus Tollens rules have graded truth values.\\
\end{remark}
\begin{theorem}If $P,Q \in \mathscr{P}$ and $e(P)=v_{i0}; e(Q)= v_{j0}$, then the truth values of the Modus Ponens and Modus Tollens are as  follows:
\begin{enumerate}
 \item $e\{(P\bigwedge(P\rightarrow Q))\rightarrow Q\}= 
 \begin{cases}
 v_{n1},&\text{ if} \ (i \geq j)\\ 
v_{(n-j+i)1} &\text{if}\  i\leq j \leq 2i \\ 
v_{(n-i)1} &\text{if} \  j\geq 2i \\
\end{cases}$\\
\item  $e\{(\neg Q)\bigwedge(P\rightarrow Q) \rightarrow(\neg P)\}= 
 \begin{cases}
 v_{n1},&\text{ if} \ (i\geq j)\\ 
v_{(n-j+i)1} &\text{if}~i \leq j~and~\  2j\leq (n+i)\\ 
v_{l1} &\text{if}~i \leq j~and~ \ 2j\geq (n+i)\\
\end{cases}$\\
\end{enumerate}
\end{theorem}
\begin{proof}The proof for  case1a and 2a i.e for $i\geq j$ is similar to the case-1a and 2a of  theorem 3.1. We will check for the other cases.\\
Case1b:
Let $i<j$.\\
$e(P\rightarrow Q)= v_{(n-j+i)1}$.\\
So, $e(P\wedge (P\rightarrow Q))= 
e(P)\wedge e(P\rightarrow Q)= \begin{cases}
v_{i0} &\text{if}\  j\leq 2i\\ 
v_{(j-i)0} &\text{if} \ j\geq 2i\\
\end{cases}$\\
$e(P\wedge(P\rightarrow Q)\rightarrow Q)= e(P\wedge(P\rightarrow Q))\rightarrow e(Q)= 
\begin{cases}
v_{(n-j+i)1} &\text{if}\  j\leq 2i\\ 
v_{(n-i)1} &\text{if} \ j\geq 2i\\
\end{cases}$\\
Case2b:Let $i<j$.\\
$e(\neg Q\wedge (P\rightarrow Q))= 
e(\neg Q)\wedge e(P\rightarrow Q)= \begin{cases}
v_{j1} &\text{if}\  2j\leq (n+i)\\ 
v_{(n-j+i)1} &\text{if} \ 2j\geq (n+i)\\
\end{cases}$\\
$e(\neg Q\wedge(P\rightarrow Q)\rightarrow \neg P)= e(\neg Q\wedge(P\rightarrow Q))\rightarrow e(\neg P)= \begin{cases}
v_{(n-j+i)1} &\text{if}\  2j\leq (n+i)\\ 
v_{j1} &\text{if} \ 2j\geq (n+i)\\
\end{cases}$\\
\end{proof}
\begin{remark} If the truth evaluation of the propositions $P$ and $Q$ is restricted to $V_0  = \{v_{00},v_{10},v_{20}..\\..,v_{n0}\}$ and $e(P)$ is greater than or equal to  $e(Q)$ then both Modus Ponens and Modus Tollens rules are absolutely true.\\
If $ e(P)= e(Q)=v_{00}$ then $P$ and $Q$ are classical propositions and Modus Ponens  and  Modus Tollens rules are absolutely true in such cases.\\
However, if  $e(P)$ is less than or equal to $e(Q)$ then both Modus Ponens and Modus Tollens rules have graded truth values.
\end{remark}
\begin{theorem}If $P,Q \in \mathscr{P}$ and $e(P)=v_{i1}; e(Q)= v_{j0}$, then the truth values of the Modus Ponens and Modus Tollens are as  follows:
\begin{enumerate}
 \item $e\{(P\bigwedge(P\rightarrow Q))\rightarrow Q\}= 
 \begin{cases}
 v_{n1},&\text{ if}\ ((i+j) \leq n)\\ 
v_{i1},&\text{ if} \ (i+j)  \geq n~and~\ n\leq (i+\dfrac{j}{2})\\
v_{(2n-i-j)1},&\text{ if}\ (i+j)  \geq n~and~\ n\geq (i+\dfrac{j}{2})
\end{cases}$\\
\item  $e\{(\neg Q)\bigwedge(P\rightarrow Q) \rightarrow(\neg P)\}= 
 \begin{cases}
 v_{n1},&\text{ if} \ ((i+j) \leq n)\\ 
v_{j1},&\text{ if}\ n\leq (j+\dfrac{i}{2})\\
v_{(2n-i-j)1},&\text{ if}\ n\geq (j+\dfrac{i}{2})
\end{cases}$\\
\end{enumerate}
\end{theorem}
\begin{proof} Case 1a: Let $(i+j)\leq n$,\\
$ e(P\rightarrow Q)= v_{00}$.\\
So, $e(P\wedge (P\rightarrow Q))= 
e(P)\wedge e(P\rightarrow Q)= v_{(n-i)0}$\\
$e(P\wedge(P\rightarrow Q)\rightarrow Q)= e(P\wedge(P\rightarrow Q))\rightarrow e(Q)= v_{n1}$\\
Case 1b:Let $(i+j)  \geq n$,\\
$e(P\rightarrow Q)= v_{(i+j-n)0}$.\\
So, $e(P\wedge (P\rightarrow Q))= 
e(P)\wedge e(P\rightarrow Q)= 
\begin{cases}
v_{(i+j-n)0},\ n\leq (i+\dfrac{j}{2})\\
v_{(n-i)0},\ n\geq (i+\dfrac{j}{2})
\end{cases}$\\
$e(P\wedge(P\rightarrow Q)\rightarrow Q)= e(P\wedge(P\rightarrow Q))\rightarrow e(Q)= 
\begin{cases}
v_{i1},\ n\leq (i+\dfrac{j}{2})\\
v_{(2n-i-j)1},\ n\geq (i+\dfrac{j}{2})
\end{cases}$\\ 
Case 2a:Let $(i+j)\leq n$,\\
 $e(\neg Q\wedge (P\rightarrow Q))= 
e(\neg Q)\wedge e(P\rightarrow Q)= v_{(n-j)0}$\\
 $e(\neg Q\wedge(P\rightarrow Q)\rightarrow \neg P)= e(\neg Q\wedge(P\rightarrow Q))\rightarrow e(\neg P)= v_{n1}$.\\
Case 2b: Let $(i+j)  \geq n$,\\
 $e(\neg Q\wedge (P\rightarrow Q))= 
e(\neg Q)\wedge e(P\rightarrow Q)= 
\begin{cases}
v_{(i+j-n)0}, \ n\leq (j+\dfrac{i}{2})\\
v_{(n-j)0},\ n\geq (j+\dfrac{i}{2})
\end{cases}$\\
 $e(\neg Q\wedge(P\rightarrow Q)\rightarrow \neg P)= e(\neg Q\wedge(P\rightarrow Q))\rightarrow e(\neg P)= \begin{cases}
v_{j1},\ n\leq (j+\dfrac{i}{2})\\
v_{(2n-i-j)1},\ n\geq (j+\dfrac{i}{2})
\end{cases}$\\ \
\end{proof}
\begin{remark}If $ e(P)= v_{n1}$ and $e(Q)=v_{00}$ then $P$ and $Q$ are classical propositions having truth values absolutely true and absolutely false. Theorem 3.3 establishes the fact that Modus Ponens  and  Modus Tollens rules are absolutely true in such cases.
\end{remark}
\begin{theorem}If $P,Q \in \mathscr{P}$ and $e(P)=v_{i0}; e(Q)= v_{j1}$, then the truth values of the Modus Ponens and Modus Tollens are as  follows:
\begin{enumerate}
 \item $e\{(P\bigwedge(P\rightarrow Q))\rightarrow Q\}= 
 \begin{cases}
 v_{n1},&\text{ if}\ ((i+j) \geq n)\\ 
v_{(i+j)1},&\text{ if} \ (i+j)  \leq n~and~\ n\leq (2i+j)\\
v_{(n-i)1},&\text{ if}\ (i+j)  \leq n~and~\ n\geq (2i+j)
\end{cases}$\\
\item  $e\{(\neg Q)\bigwedge(P\rightarrow Q) \rightarrow(\neg P)\}= 
 \begin{cases}
 v_{n1},&\text{ if} \ ((i+j) \geq n)\\ 
v_{(i+j)1},&\text{ if}\  (i+j)  \leq n~and~\ n \leq (2j+i)\\
v_{(n-j)1},&\text{ if}\  (i+j)  \leq n~and~\ n\geq (2j+i)
\end{cases}$\\
\end{enumerate}
\end{theorem}
\begin{proof}Case 1a: Let $(i+j)\geq n$,\\
$ e(P\rightarrow Q)= v_{n1}$.\\
So, $e(P\wedge (P\rightarrow Q))= 
e(P)\wedge e(P\rightarrow Q)= v_{i0}$\\
$e(P\wedge(P\rightarrow Q)\rightarrow Q)= e(P\wedge(P\rightarrow Q))\rightarrow e(Q)= v_{n1}$\\
Case1b:Let $(i+j)\leq n$,\\
$e(P\rightarrow Q)= v_{(i+j)1}$.\\
So, $e(P\wedge (P\rightarrow Q))= 
e(P)\wedge e(P\rightarrow Q)= 
\begin{cases}
v_{i0}, \ n\leq (2i+j)\\
v_{(n-i-j)0},\ n\geq (2i+j)
\end{cases}$\\
$e(P\wedge(P\rightarrow Q)\rightarrow Q)= e(P\wedge(P\rightarrow Q))\rightarrow e(Q)= 
\begin{cases}
v_{(i+j)1}, \ n\leq (2i+j)\\
v_{(n-i)1}, \ n\geq (2i+j)
\end{cases}$\\
Case2a: Let $(i+j)\geq n$,\\
 $e(\neg Q\wedge (P\rightarrow Q))= 
e(\neg Q)\wedge e(P\rightarrow Q)= v_{j0}$\\
$e(\neg Q\wedge(P\rightarrow Q)\rightarrow \neg P)= e(\neg Q\wedge(P\rightarrow Q))\rightarrow e(\neg P)= v_{n1}$.\\
Case 2b: Let $(i+j)\leq n$,\\
$e(\neg Q\wedge (P\rightarrow Q))= 
e(\neg Q)\wedge e(P\rightarrow Q)= 
\begin{cases}
v_{j0}, \ n\leq (2j+i)\\
v_{(n-i-j)0},\ n\geq (2j+i)
\end{cases}$\\
 $e(\neg Q\wedge(P\rightarrow Q)\rightarrow \neg P)= e(\neg Q\wedge(P\rightarrow Q))\rightarrow e(\neg P)= 
\begin{cases}
v_{(i+j)1}, \ n\leq (2j+i)\\
v_{(n-j)1}, \ n\geq (2j+i)
\end{cases}$
\end{proof}
\begin{remark}If $ e(P)= v_{00}$ and $e(Q)=v_{n1}$ then $P$ and $Q$ are classical propositions having truth values absolutely true and absolutely false. Theorem 3.4 establishes the fact that Modus Ponens  and  Modus Tollens rules are absolutely true in such cases.
\end{remark}
Let us suppose that  the linguistic  truth values of a particular set of LTVP may be represented by the following lattice structure  (Fig 3). Truth values of the inference rules (ModusPonens and Modus Tollens) for a few  cases  are computed in the examples below .
\begin{figure}[h]
\centering
\begin{minipage}[b]{3in}
\hfill
  \includegraphics[width=3.0in, height=3.0in]{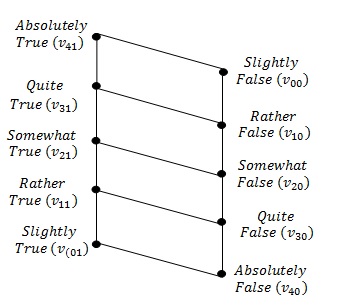}
\caption{Hasse Diagram of Linguistic Truth Value set V having five values}
\end{minipage}
 \end{figure} 
\begin{example} 
Modus Ponens and Modus Tollens rules are Quite True if an LTVP P is Quite True and $P\rightarrow Q$ where Q is Rather True.\\
Let $P,Q \in \mathscr{P}$ and $e(P) = Quite~True = v_{31}; e(Q) = Rather~True = v_{21}.$\\
Here we have taken $e(P) > e(Q)$ i.e. the case when $i > j$.
$e(P\rightarrow Q) = Somewhat~True$.\\
Then $e(P\wedge (P\rightarrow Q)) = Somewhat~True$ and $e(\neg Q)\wedge e(P\rightarrow Q) = Somewhat~False$.\\
Therefore, $e(P\wedge (P\rightarrow Q)\rightarrow Q) = Quite~True$ \\ and $e(\neg Q\wedge (P\rightarrow Q)\rightarrow \neg P) = Quite~True$.
\end{example}
\begin{example}
Modus Ponens rule is Somewhat True and Modus Tollens rule is Absolutely True if an LTVP P is Somewhat False and $P\rightarrow Q$ where Q is Absolutely False.\\
Let $P,Q \in \mathscr{P}$ and $e(P) = Somewhat~False = v_{20}; e(Q) = Absolutely~False = v_{40}.$\\
Here we have taken $e(Q) < e(P)$ i.e. the case when $i < j$.
$e(P \rightarrow Q) = Somewhat~True.$\\
Then $e(P \wedge (P\rightarrow Q)) = Somewhat~False$ and $e(\neg Q)\wedge e(P\rightarrow Q) = Somewhat~True$.\\
Therefore, $e(P \wedge (P \rightarrow Q) \rightarrow Q) = Somewhat~True$\\
and $e(\neg Q \wedge (P \rightarrow Q) \rightarrow \neg P) = Absolutely~True.$
\end{example}
 
\begin{example}
Modus Ponens rule is Somewhat True and Modus Tollens rule is Absolutely True if an LTVP P is Somewhat True and $P\rightarrow Q$ where Q is Absolutely False.\\
Let $P,Q \in \mathscr{P}$ and $e(P)= Somewhat~True=v_{(2)1}$; $e(Q)=  Absolutely~False= v_{(4)0}$.\\
Here we have taken the case when $(i+j)>n$.\\
$e(P\rightarrow Q)= Somewhat~False$\\
Then $e(P\wedge (P\rightarrow Q))= Somewhat~False$
and $e(\neg Q)\wedge e(P\rightarrow Q)= Somewhat~False$\\
Therefore, $e(P\wedge(P\rightarrow Q)\rightarrow Q)= Somewhat~True$\\
and $e(\neg Q\wedge(P\rightarrow Q)\rightarrow \neg P)= Absolutely~True.$\\
\end{example}

\begin{example}
Modus Ponens rule is Absolutely True and Modus Tollens rule is Somewhat True if an LTVP P is Slightly False and $P\rightarrow Q$ where Q is Somewhat True.\\
Let $P,Q \in \mathscr{P}$ and $e(P)= Slightly~False=v_{(0)0}$; $e(Q)=  Somewhat~True= v_{(2)1}$.\\
Here we have taken the case when $(i+j)<n$.\\
$e(P\rightarrow Q)= Somewhat~True$\\
Then $e(P\wedge (P\rightarrow Q))= Somewhat~False$
and $e(\neg Q)\wedge e(P\rightarrow Q)= Somewhat~False$\\
Therefore, $e(P\wedge(P\rightarrow Q)\rightarrow Q)= Absolutely~True$\\
and $e(\neg Q\wedge(P\rightarrow Q)\rightarrow \neg P)= Somewhat~True.$\\
\end{example}
\section{Reasoning In Quasi-Linguistic Truth-valued Propositions}
 Truth values of Modus Ponens and Modus Tollens rules for Quasi-Linguistic truth-valued Propositions (QLTVP) are computed in this section. It is observed that these fundamental reasoning tools are not always absolutely true.\\
In this section $V$, the set of linguistic truth values have a lattice structure represented by Fig.2 (i.e $v_{io}~ and ~v_{(n-i)1}$ are non comparable).Let $\mathscr{Q}$ = \{Q/ Q is a QLTVP\} be the set of all QLTVPs. The \textit{truth evaluation e} is a function given by $e: \mathscr{Q} \rightarrow V$.
\begin{theorem}
If $P,Q \in \mathscr{Q}$ and $e(P)=v_{k1}; e(Q)= v_{l1}$, then the truth values of the Modus Ponens and Modus Tollens are as  follows:
\begin{enumerate}
 \item $e\{(P\bigwedge(P\rightarrow Q))\rightarrow Q\}= 
 \begin{cases}
 v_{n1},&\text{ if} \ (k \leq l)\\ 
v_{(n-k+l)1} &\text{if}\  (k \geq l)~and~ 2k\leq (n+l)\\ 
v_{k1} &\text{if} \  (k \geq l)~and~2k\geq (n+l)\\
\end{cases}$\\
\item  $e\{(\neg Q)\bigwedge(P\rightarrow Q) \rightarrow(\neg P)\}= 
 \begin{cases}
 v_{n1},&\text{ if} \ (k \leq l)\\ 
v_{(n-k+l)1} &\text{if}\  l \leq  k\leq 2l~and~(k-l)\neq i\\ 
v_{(n-l)1} &\text{if} \  k\geq 2l~and~(k-l)\neq i\\
v_{(n-k+l)1} &\text{if}~k>l~and~2l > (k+1)~and~ (k-l)=i\\ 
v_{(n-l+1)1} &\text{if}~k>l~and~ 2l\leq (k+1)~and~ (k-l)=i\\
\end{cases}$\\
\end{enumerate}
\end{theorem}
\begin{proof} Case 1a: Let $k\leq l$,\\
$ e(P\rightarrow Q)= v_{n1}$.\\
So, $e(P\wedge (P\rightarrow Q))= 
e(P)\wedge e(P\rightarrow Q)= v_{k1}$\\
Thus, $e(P\wedge(P\rightarrow Q)\rightarrow Q)= e(P\wedge(P\rightarrow Q))\rightarrow e(Q)= v_{n1}$\\
Case1b:Let $k>l$,\\
$e(P\rightarrow Q)= v_{(n-k+l)1}$.\\
So, $e(P\wedge (P\rightarrow Q))= 
e(P)\wedge e(P\rightarrow Q)= \begin{cases}
v_{k1} &\text{if}\  2k\leq (n+l)\\ 
v_{(n-k+l)1} &\text{if} \ 2k\geq (n+l)\\
\end{cases}$\\
Thus, $e(P\wedge(P\rightarrow Q)\rightarrow Q)= e(P\wedge(P\rightarrow Q))\rightarrow e(Q)=\begin{cases} v_{(n-k+l)1} &\text{if}\  2k\leq (n+l)\\ 
v_{k1} &\text{if} \ 2k\geq (n+l)\\
\end{cases}$\\
Case 2a: Let $k\leq l$,\\
$e(\neg Q\wedge (P\rightarrow Q))= 
e(\neg Q)\wedge e(P\rightarrow Q)= v_{l0}$\\
so,$e(\neg Q\wedge(P\rightarrow Q)\rightarrow \neg P)= e(\neg Q\wedge(P\rightarrow Q))\rightarrow e(\neg P)= v_{n1}$.\\
Case2b: If $ k >l ~and~ (k-l)  \neq i $\\
$e(\neg Q\wedge (P\rightarrow Q))= 
e(\neg Q)\wedge e(P\rightarrow Q)= \begin{cases}
v_{l0} &\text{if}\   k\leq 2l\\ 
v_{(k-l)0} &\text{if} \  2l\leq k\\
\end{cases}$\\
and $e(\neg Q\wedge(P\rightarrow Q)\rightarrow \neg P)= e(\neg Q\wedge(P\rightarrow Q))\rightarrow e(\neg P)= \begin{cases}
v_{(n-k+l)1} &\text{if}\   k\leq 2l\\ 
v_{(n-l)1} &\text{if} \  k\geq 2l\\
\end{cases}$\\
Case2c: If $k >l ~and~ (k-l)  = i $\\
$e(\neg Q\wedge (P\rightarrow Q))= e(\neg Q)\wedge e(P\rightarrow Q)= v_{l0}\wedge v_{(n-i)1}= \begin{cases}
v_{l0} &\text{if}~ 2l > (k+1)\\ 
v_{(k-l+1)0} &\text{if} \ 2l\leq (k+1)\\
\end{cases}$\\
$e(\neg Q\wedge (P\rightarrow Q)\rightarrow \neg P)= e(\neg Q\wedge (P\rightarrow Q))\rightarrow e(\neg P)= \begin{cases}
v_{(n-k+l)1} &\text{if}~2l > (k+1)\\ 
v_{(n-l+1)1} &\text{if}~ 2l\leq (k+1)\\
\end{cases}$\\

\end{proof}
\begin{remark} Thus we see that if the truth evaluation of the propositions $P$ and $Q$ is restricted to $V_1  = \{v_{01},v_{11},v_{21}...,v_{n1}\}$ and $e(P)$ is less than $e(Q)$ then the Modus Ponens rule and the Modus Tollens rule is absolutely true.\\ If $e(P) = e(Q) = v_{n1}$ then $P$ and $Q$ are classical propositions and Modus Ponens and Modus
Tollens rules are absolutely true in such cases.\\
However, if  $e(P)$ is greater than or equal to $e(Q)$ then both Modus Ponens and Modus Tollens rules have graded truth values.
\end{remark}
\begin{theorem}If $P,Q \in \mathscr{Q}$ and $e(P)=v_{k0}; e(Q)= v_{l0}$, then the truth values of the Modus Ponens and Modus Tollens are as  follows:
\begin{enumerate}
 \item $e\{(P\bigwedge(P\rightarrow Q))\rightarrow Q\}= 
 \begin{cases}
 v_{n1},&\text{ if} \ (k \geq l)\\ 
v_{(n-l+k)1} &\text{if}\  k<  l \leq 2k~and (l-k) \neq i \\ 
v_{(n-k)1} &\text{if} \  l\geq 2k ~and (l-k) \neq i\\
v_{(n-l+k)1} &\text{if}\ k<  l~and~ 2k> (l+1)~and~(l-k)=i\\ 
v_{(n-k+1)1} &\text{if} \ k<  l~and~2k \leq (l+1)~and (l-k)=i\\
\end{cases}$\\
\item  $e\{(\neg Q)\bigwedge(P\rightarrow Q) \rightarrow(\neg P)\}= 
 \begin{cases}
 v_{n1},&\text{ if} \ (k \geq l)\\ 
v_{(n-l+k)1} &\text{if}\  2l\leq (n+k)\\ 
v_{l1} &\text{if} \ 2l\geq (n+k)\\
\end{cases}$\\
\end{enumerate}
\end{theorem}
\begin{proof}The proof for the first case i.e for $k\geq l$ in case1a,and case 2a is similar to the case-1a and 2a of  theorem 4.1. We will check for the case when $k<l$.\\
Case 1b: $k< l,~(l-k)\neq i $\\
$e(P\rightarrow Q)= v_{(n-l+k)1}$.\\
$e(P \wedge (P\rightarrow Q))= 
e(P) \wedge e(P\rightarrow Q)= \begin{cases}
v_{k0} &\text{if}\   l\leq 2k\\ 
v_{(l-k)0} &\text{if} \  2k\leq l\\
\end{cases}$\\
and $e(P \wedge(P\rightarrow Q)\rightarrow Q)= e(P\wedge(P\rightarrow Q))\rightarrow e(Q)= \begin{cases}

v_{(n-l+k)1} &\text{if}\   l\leq 2k\\ 
v_{(n-k)1} &\text{if} \  l\geq 2k\\
\end{cases}$\\
Case 1c:\\ If $k< l,~ (l-k)= i $\\
$e(P \wedge (P\rightarrow Q))= e(P)\wedge e(P\rightarrow Q)= \begin{cases}
v_{k0} &\text{if}\ 2k> (l+1)\\ 
v_{(l-k+1)0} &\text{if} \ 2k\leq (l+1)\\
\end{cases}$\\
$e(P \wedge (P\rightarrow Q)\rightarrow Q)= e(P \wedge (P\rightarrow Q))\rightarrow e(Q)= \begin{cases}
v_{(n-l+k)1} &\text{if}\  2k> (l+1)\\ 
v_{(n-k+1)1} &\text{if} \ 2k \leq (l+1)\\
\end{cases}$\\
Case 2a:\\If $k< l $\\
 $e(\neg Q \wedge (P\rightarrow Q))= 
e(\neg Q)\wedge e(P\rightarrow Q)= \begin{cases}
v_{l1} &\text{if}\  2l\leq (n+k)\\ 
v_{(n-l+k)1} &\text{if} \ 2l\geq (n+k)\\
\end{cases}$\\
Thus, $e(\neg Q \wedge(P\rightarrow Q)\rightarrow Q)= e(\neg Q \wedge(P\rightarrow Q))\rightarrow e(\neg P)= 
\begin{cases}
v_{(n-l+k)1} &\text{if}\  2l\leq (n+k)\\ 
v_{l1} &\text{if} \ 2l\geq (n+k)\\
\end{cases}$\\
\end{proof}
\begin{remark}
If the truth evaluation of the propositions $P$ and $Q$ is restricted to $V_0  = \{v_{00},v_{10},v_{20}..., v_{n0}\}$ and $e(P)$ is greater than or equal to $e(Q)$ then both Modus Ponens and Modus Tollens rules are
absolutely true.\\
If $e(P) = e(Q) = v_{00}$ then $P$ and $Q$ are classical propositions and Modus Ponens and Modus
Tollens rules are absolutely true in such cases.\\
However, if  $e(P)$ is less than or equal to $e(Q)$ then both Modus Ponens and Modus Tollens rules have graded truth values.
\end{remark}
\begin{theorem}If $P,Q \in \mathscr{Q}$ and $e(P)=v_{k1}; e(Q)= v_{l0}$, then the truth values of the Modus Ponens and Modus Tollens are as  follows:
\begin{enumerate}
 \item $e\{(P\bigwedge(P\rightarrow Q))\rightarrow Q\}= 
 \begin{cases}
 v_{n1},&\text{if}\ ((k+l) \leq n ~and ~k \neq (n-i))\\ 
v_{n1},&\text{if}\ ((k+l) \leq n ~and ~k = (n-i) \ and \ l \leq (i+1))\\
v_{n-l+i+1},&\text{if}\ ((k+l) \leq n ~and ~k = (n-i) \ and \ l \geq (i+1))\\
v_{k1}, &\text{if}\ ((k+l) >n ~ and ~k \neq (n-i) \ and \ n \leq (k+\dfrac{l}{2}))\\
v_{(2n-k-l)1}, &\text{if}\ ((k+l) >n ~and ~k \neq (n-i) \ and \ n\geq (k+\dfrac{l}{2}))\\ 
 v_{n1},&\text{if}\ ((k+l) =(n+1) ~ and ~k = (n-i))\\
 v_{(2n-k-l+1)1},&\text{if}\ ((k+l) > (n+1) ~ and ~k = (n-i) \ and \ (n-k) \geq \dfrac{l-1}{2})\\
v_{k1}, &\text{if}\ ((k+l) >n ~and~k = (n-i) \ and \ (n-k) \leq \dfrac{l-1}{2})
\end{cases}$\\
\item  $e\{(\neg Q)\bigwedge(P\rightarrow Q) \rightarrow(\neg P)\}= 
 \begin{cases}
 v_{n1},\ if \ ((k+l) \leq n ~and ~l \neq (n-i))\\ 
v_{n1},\ if \  ((k+l) \leq n ~and ~l = (n-i)~and~k \leq (n-l+1)\\
v_{(2n-k-l+1)1}, \ if \ (k+l) \leq n ~and ~l = (n-i)~and~k\geq (n-l+1)\\
v_{l1},~if~((k+l) \leq n ~and ~l \neq (n-i)~and~n\leq (l+\dfrac{k}{2})\\
v_{(2n-k-l)1}, ~if~((k+l) \leq n ~and ~l \neq (n-i)~and~n\geq (l+\dfrac{k}{2}) \\
v_{l1},~if~(k+l)>n, ~l = (n-i), ~and~n\leq (l+\dfrac{k-1}{2})\\
v_{(2n-k-l+1)1}, ~if~(k+l)>(n+1), \\~and~l = (n-i)~and~n\geq (l+\dfrac{k-1}{2}) \\
v_{n1}, ~if~(k+l)=(n+1), ~and ~\\ l = (n-i)~and~n\geq (l+\dfrac{k-1}{2}) \\
v_{(2n-k-l)1}, ~if~((k+l) \geq n ~and ~l \neq (n-i)~and~n\geq (l+\dfrac{k}{2}) \\
\end{cases}$\\
\end{enumerate}

\end{theorem}
\begin{proof} 
Case 1a: Let $(k+l)\leq n,~ k \neq (n-i)$\\
$ e(P\rightarrow Q)= v_{00}$.\\
So, $e(P\wedge (P\rightarrow Q))= 
e(P)\wedge e(P\rightarrow Q)= v_{(n-k)0}$\\
Thus, $e(P\wedge(P\rightarrow Q)\rightarrow Q)= e(P\wedge(P\rightarrow Q))\rightarrow e(Q)= v_{n1}$\\
Case 1b: Let $(k+l)\leq n,~ k = (n-i),l= (n-i)$\\
$ e(P\rightarrow Q)= v_{00}$.\\
So, $e(P\wedge (P\rightarrow Q))= 
e(P)\wedge e(P\rightarrow Q)= v_{(i+1)0}$\\
Thus, $e(P\wedge(P\rightarrow Q)\rightarrow Q)= e(P\wedge(P\rightarrow Q))\rightarrow e(Q)= \begin{cases}v_{n1},\ if \ l\leq (i+1)\\
v_{(n-l+i+1)1}, \ if \ l\geq (i+1)
\end{cases}$\\
Case1c: Let $(k+l)>n, ~k \neq (n-i)$\\
$e(P\rightarrow Q)= v_{(k+l-n)0} $.\\
So, $e(P\wedge (P\rightarrow Q))= 
e(P)\wedge e(P\rightarrow Q)= 
\begin{cases}
v_{(k+l-n)0}, ~if~n\leq (k+\dfrac{l}{2})\\
v_{(n-k)0}, ~if~n\geq (k+\dfrac{l}{2}) 
\end{cases}$\\
Thus, $e(P\wedge(P\rightarrow Q)\rightarrow Q)= e(P\wedge(P\rightarrow Q))\rightarrow e(Q)= 
\begin{cases}
v_{k1},~if~n\leq (k+\dfrac{l}{2})\\
v_{(2n-k-l)1},~if~n\geq (k+\dfrac{l}{2}) 
\end{cases}$\\ 
Case1d:Let $(k+l)>n, ~k = (n-i)$\\
$e(P\rightarrow Q)= v_{(k+l-n)0} $.\\
So, $e(P\wedge (P\rightarrow Q))= 
e(P)\wedge e(P\rightarrow Q)= 
\begin{cases}
v_{(i+1)0}, ~if ~n\geq (k+\dfrac{l-1}{2})\\
v_{(l-i)0}, ~if~n\leq (k+\dfrac{l-1}{2}) 
\end{cases}$\\
Thus, $e(P\wedge(P\rightarrow Q)\rightarrow Q)= \\e(P\wedge(P\rightarrow Q))\rightarrow e(Q)= 
\begin{cases}
v_{n1}, ~if~ n\geq (k+\dfrac{l-1}{2})~and~(k+l)=(n+1)\\
v_{(2n-k-l+1)1},~if~n\geq (k+\dfrac{l-1}{2})~and~(k+l)>(n+1))\\
v_{k1},~if~n\leq (k+\dfrac{l-1}{2}) 
\end{cases}$\\ 
Case2a:Let $(k+l)\leq n, l\neq (n-i)$\\
$e(\neg Q\wedge (P\rightarrow Q))= 
e(\neg Q)\wedge e(P\rightarrow Q)= v_{(n-l)0}$\\
and $e(\neg Q\wedge(P\rightarrow Q)\rightarrow \neg P)= e(\neg Q\wedge(P\rightarrow Q))\rightarrow e(\neg P)= v_{n1}$\\
Case 2b: Let $(k+l)\leq n,~l= (n-i)$\\
and $e(\neg Q\wedge (P\rightarrow Q))= 
e(\neg Q)\wedge e(P\rightarrow Q)= v_{(i+1)0}$\\
and $e(\neg Q\wedge(P\rightarrow Q)\rightarrow \neg P)= e(\neg Q\wedge(P\rightarrow Q))\rightarrow e(\neg P)=
\begin{cases} 
v_{n1},\ if \ k \leq (n-l+1)\\
v_{(2n-k-l+1)1}, \ if \ k\geq (n-l+1)
\end{cases}$\\
Case 2c: Let $(k+l)>n, ~l \neq (n-i)$\\
 $e(\neg Q\wedge (P\rightarrow Q))= 
e(\neg Q)\wedge e(P\rightarrow Q)= 
\begin{cases}
v_{(k+l-n)0},~if~n\leq (l+\dfrac{k}{2})\\
v_{(n-l)0}, ~if~n\geq (l+\dfrac{k}{2}) 
\end{cases}$\\ 
and $e(\neg Q\wedge(P\rightarrow Q)\rightarrow \neg P)=e(\neg Q\wedge(P\rightarrow Q))\rightarrow e(\neg P)= 
\begin{cases}
v_{l1},~if~n\leq (l+\dfrac{k}{2})\\
v_{(2n-k-l)1}, ~if~n\geq (l+\dfrac{k}{2}) 
\end{cases}$\\ 
Case 2d: Let $(k+l)>n, ~l = (n-i)$\\
and $e(\neg Q\wedge (P\rightarrow Q))= 
e(\neg Q)\wedge e(P\rightarrow Q)= 
\begin{cases}
v_{(k+l-n)0},~if~(k+l)>n, ~l = (n-i)~n\leq (l+\dfrac{k-1}{2})\\
v_{(n-l+1)0}, ~if~(k+l)>n, ~l = (n-i)~n\geq (l+\dfrac{k-1}{2}) 
\end{cases}$\\ 
$e(\neg Q\wedge(P\rightarrow Q)\rightarrow \neg P)=e(\neg Q\wedge(P\rightarrow Q))\rightarrow e(\neg P)= 
\begin{cases}
v_{l1},~if~(k+l)>n, ~and~l = (n-i),\\and~n\leq (l+\dfrac{k-1}{2})\\
v_{(2n-k-l+1)1}, ~if~(k+l)>(n+1), \\ and~l = (n-i)~and~n\geq (l+\dfrac{k-1}{2}) \\
v_{n1}, ~if~(k+l)=(n+1), \\and~l = (n-i)~and~n\geq (l+\dfrac{k-1}{2}) 
\end{cases}$\\ 
Case2e: Let $(k+l)>n, ~l \neq (n-i)$\\
and $e(\neg Q\wedge (P\rightarrow Q))=\\ 
e(\neg Q)\wedge e(P\rightarrow Q)= 
\begin{cases}
v_{(n-l)0}, ~if~(k+l)>n, ~l \neq (n-i)~n\geq (l+\dfrac{k}{2})\\
v_{(k+l-n)0},~if~(k+l)>n, ~l \neq (n-i)~n\leq (l+\dfrac{k}{2})\\
\end{cases}$\\
$e(\neg Q\wedge(P\rightarrow Q)\rightarrow \neg P)=\\e(\neg Q\wedge(P\rightarrow Q))\rightarrow e(\neg P)= 
\begin{cases}
v_{l1},~if~(k+l)>n, ~and~l \neq (n-i),\\and~n\leq (l+\dfrac{k}{2})\\
v_{(2n-k-l)1}, ~if~(k+l)>(n), ~and~l \neq (n-i)~and~n\geq (l+\dfrac{k}{2}) \\
v_{n1}, ~if~(k+l)=(n+1), ~and~l = (n-i)~and~n\geq (l+\dfrac{k-1}{2}) 
\end{cases}$\\ 
\end{proof}
\begin{remark}If $P,Q \in \mathscr{Q}$ and $e(P)=v_{k1}, e(Q)= v_{l0}$, then Modus Ponens and Modus Tollens have graded truth values.
\end{remark}
\begin{theorem}If $P,Q \in \mathscr{Q}$ and $e(P)=v_{k0}; e(Q)= v_{l1}$, then the truth values of the Modus Ponens and Modus Tollens are as  follows:
\begin{enumerate}
 \item $e\{(P\bigwedge(P\rightarrow Q))\rightarrow Q\}= 
 \begin{cases}
 v_{n1},\ if \ (k+l) \geq n\\ 
v_{(k+l)1},\ if \ ((k+l) \leq n ~and ~(k+l) \neq (n-i) \ and \ n \leq (2k+l))\\
v_{(n-k)1}, ~if~((k+l) \leq n ~ and ~(k+l) \neq (n-i) \ and \ n \geq (2k+l))\\
v_{(k+l)1},~if~((k+l) \leq n ~ and ~(k+l) = (n-i) \ and\  n \leq (2k+l))\\
v_{n1},~if~((k+l) \leq n ~ and ~(k+l) = (n-i) \ and\ n \geq (2k+l)~and~0 \leq k \leq1)\\
v_{(n-k+1)1},~if~((k+l) \leq n ~ and ~(k+l) \neq (n-i) \ and\  n \geq (2k+l))
\end{cases}$ \\
\item  $e\{(\neg Q)\bigwedge(P\rightarrow Q) \rightarrow(\neg P)\}= 
 \begin{cases}
 v_{n1},\ if \ (k+l) \geq n\\ 
v_{(k+l)1},~if~((k+l) \leq n ~and ~l \neq (n-i)~and~ n \leq (2l+k))\\
v_{(n-l)1},~if~((k+l) \leq n ~and ~l \neq (n-i)~and~ n \geq (2l+k))\\
v_{(k+l)1},~if~((k+l) \leq n ~and ~l = (n-i)~and~\ if \ n \leq (2l+k-1))\\
v_{n1},~if~((k+l) \leq n ~and ~l = (n-i)~and~ n \geq (2l+k-1)\\~and~ 0 \leq l \leq1)\\
v_{(n-l+1)1},~if~((k+l) \leq n ~and ~l = (n-i)~and~\ if \ n \geq (2l+k-1)\\ ~and~  1 <l \leq \dfrac{(n-k+1)}{2})\\
\end{cases}$\\
\end{enumerate}
\end{theorem}
\begin{proof} Case 1a: Let $(k+l)\geq n$\\
$ e(P\rightarrow Q)= v_{n1}$.\\
So, $e(P\wedge (P\rightarrow Q))= 
e(P)\wedge e(P\rightarrow Q)= v_{k0}$\\
$e(P\wedge(P\rightarrow Q)\rightarrow Q)= e(P\wedge(P\rightarrow Q))\rightarrow e(Q)= v_{n1}$\\
Case1b: Let $(k+l) \leq n,\ (k+l) \neq (n-i)$\\
$e(P\rightarrow Q)= v_{(k+l)1}$.\\
So, $e(P\wedge (P\rightarrow Q))= 
e(P)\wedge e(P\rightarrow Q)= 
\begin{cases}
v_{k0},\ if \ n \leq (2k+l)\\
v_{(n-k-l)0},\ if \ n \geq (2k+l)
\end{cases}$\\
 $e(P\wedge(P\rightarrow Q)\rightarrow Q)= e(P\wedge(P\rightarrow Q))\rightarrow e(Q)= 
\begin{cases}
v_{(k+l)1},\ if \ n \leq (2k+l)\\
v_{(n-k)1},\ if \ n \geq (2k+l)
\end{cases}$ \\
Case1c: Let $(k+l) \leq n,\ (k+l) = (n-i)$\\
$e(P\rightarrow Q)= v_{(k+l)1}$.\\
So, $e(P\wedge (P\rightarrow Q))= 
e(P)\wedge e(P\rightarrow Q)= 
\begin{cases}
v_{k0},\ if \ n \leq (2k+l)\\
v_{(n-k-l+1)0},\ if \ n \geq (2k+l)
\end{cases}$\\
 $e(P\wedge(P\rightarrow Q)\rightarrow Q)= e(P\wedge(P\rightarrow Q))\rightarrow e(Q)= 
\begin{cases}
v_{(k+l)1},\ if \ n \leq (2k+l)\\
v_{n1},\ if \ n \geq (2k+l)~and~0 \leq k \leq1\\
v_{(n-k+1)1},\ if \ n \geq (2k+l)
\end{cases}$ \\
Case2a: Let $(k+l)\geq n$\\
$e(\neg Q\wedge (P\rightarrow Q))= 
e(\neg Q)\wedge e(P\rightarrow Q)= v_{l0}$\\
$e(\neg Q\wedge(P\rightarrow Q)\rightarrow \neg P)= e(\neg Q\wedge(P\rightarrow Q))\rightarrow e(\neg P)= v_{n1}$.\\
Case2b: Let $(k+l) \leq n,\ (k+l) \neq (n-i)$\\
$e(P\rightarrow Q)= v_{(k+l)1}$.\\
$e(\neg Q \wedge (P\rightarrow Q))= 
e(\neg Q)\wedge e(P\rightarrow Q)= 
\begin{cases}
v_{l0},\ if \ n \leq (2l+k)\\
v_{(n-k-l)0},\ if \ n \geq (2l+k)
\end{cases}$\\
$e(\neg Q\wedge(P\rightarrow Q)\rightarrow \neg P)=
 e(\neg Q\wedge(P\rightarrow Q))\rightarrow e(\neg P)= 
\begin{cases}
v_{(k+l)1},\ if \ n \leq (2l+k)\\
v_{(n-l)1},\ if \ n \geq (2l+k)
\end{cases}$ \\
Case 2c: Let $(k+l)<n,~(k+l)= (n-i)$\\
$e(P\rightarrow Q)= v_{(k+l)1}$.\\
$e(\neg Q\wedge (P\rightarrow Q))= 
e(\neg Q)\wedge e(P\rightarrow Q)= 
\begin{cases}
v_{l0},\ if \ n \leq (2l+k-1)\\
v_{(n-k-l+1)0},\ if \ n \geq (2l+k-1)
\end{cases}$\\\\
$e(\neg Q\wedge(P\rightarrow Q)\rightarrow \neg P)=\\ e(\neg Q\wedge(P\rightarrow Q))\rightarrow e(\neg P)= \begin{cases}
v_{(k+l)1},\ if \ n \leq (2l+k-1)\\
v_{n1},\ if \ n \geq (2l+k-1)~and~ 0 \leq l \leq1\\
v_{(n-l+1)1},\ if \ n \geq (2l+k-1) ~and~  1 <l \leq (n-k+1)/2\\
\end{cases} $
\end{proof}
\begin{remark}If $P,Q \in \mathscr{Q}$ and $e(P)=v_{k0}, e(Q)= v_{l1}$, then Modus Ponens and Modus Tollens have graded truth values.
 \end{remark}
We consider a set of QLTVP where the linguistic truth values of the propositions are represented  by the lattice of the figure (Fig 4) below. In the following examples we have computed the truth values of the inference rules (ModusPonens and Modus Tollens) for a few particular cases.
\begin{figure}[h]
\centering
\begin{minipage}[b]{3in}
\hfill
  \includegraphics[width=3.0in, height=3.0in]{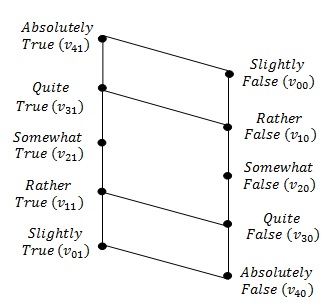}
\caption{Hasse Diagram of Quasi-Linguistic Truth-Value set V having five values}
\end{minipage}
 \end{figure} 
\begin{example}Let $P,Q \in \mathscr{Q}$ and $e(P)= Quite \ True$; $e(Q)=  Rather \ True$. Then  Modus Ponens rule is Quite True and Modus Tollens rules is Absolutely True.\\
If $P,Q \in \mathscr{Q}$ and $e(P)$= Quite true= $v_{31}; e(Q$) =  Rather True =$ v_{11}$ .\\
Here we have taken $e(P)> e(Q)$ i.e. the case when $k>l$.\\
$e(P\rightarrow Q)=v_{21}$\\
Then $e(P\wedge (P\rightarrow Q))= v_{21}$
and $e(\neg Q)\wedge e(P\rightarrow Q)=v_{30}$\\ 
Therefore, $e(P\wedge(P\rightarrow Q)\rightarrow Q)= v_{31}$=Quite  true\\
and $e(\neg Q\wedge(P\rightarrow Q)\rightarrow \neg P)=v_{41}$= Absolutely True.\\
Thus we can see that the inference rules may be not absolutely true but have some truth values close to absolutely true.
\end{example}
\begin{example}
If $P,Q \in \mathscr{Q}$ and $e(P)= v_{10}$= Rather \ False; $e(Q)=  v_{20}$=Somewhat \ False, then $e\{(P\wedge(P\rightarrow Q))\rightarrow Q\}= v_{31}$ and $e\{(\neg Q)\wedge(P\rightarrow Q) \rightarrow(\neg P)\}= v_{31}$ .\\
Here we have taken $e(P)< e(Q)$ i.e. the case when $k<l$.\\
$e(P\rightarrow Q)=v_{31}$\\
Then $e(P\wedge (P\rightarrow Q))= v_{10}$
and $e(\neg Q)\wedge e(P\rightarrow Q)= v_{21}$\\
Therefore, $e(P\wedge(P\rightarrow Q)\rightarrow Q)= v_{31}$\\
and $e(\neg Q\wedge(P\rightarrow Q)\rightarrow \neg P)= v_{31}$\\
Thus we can see that both the Modus Ponens and Modus Tollens rule are Quite True.
\end{example}
\begin{example}
If $P,Q \in \mathscr{Q}$ and $e(P)= v_{21}$= Somewhat \ True; $e(Q)=  v_{30}$=Slightly \ False, then Modus Ponens rule is Absolutely True, and Modus Tollens rule is Quite True.\\
Here we have taken the case when $(k+l)>n$ and also $k=(n-i), ~and~ l \neq (n-i)$.\\
$e(P\rightarrow Q)= v_{10}$
Then $e(P\wedge (P\rightarrow Q))= v_{30}$\\
Therefore, $e(P\wedge(P\rightarrow Q)\rightarrow Q)= v_{41}$\\
Here,  $e(\neg Q)\wedge e(P\rightarrow Q)= v_{10}$.\\ $e(\neg Q\wedge(P\rightarrow Q)\rightarrow \neg P)= v_{31}$\\
\end{example}
\begin{example}
If $P,Q \in \mathscr{Q}$ and $e(P)= v_{00}$; $e(Q)=  v_{31}$, then Modus Ponens rule is Absolutely True, and Modus Tollens rule is Quite True.\\
Here we have taken the case when $(k+l)<n, (k+l) \neq (n-i), n>(2k+l)$.\\
$e(P\rightarrow Q)= v_{31}$\\
Then $e(P\wedge (P\rightarrow Q))= v_{10}$. Therefore, $e(P\wedge(P\rightarrow Q)\rightarrow Q)= v_{41}$\\
and $e(\neg Q)\wedge e(P\rightarrow Q)= v_{30}$ and $e(\neg Q\wedge(P\rightarrow Q)\rightarrow \neg P)=v_{31}$\\
\end{example}
\section{Conclusion}In this paper we have computed truth values of Modus Ponens and Modus Tollens rule for linguistic truth valued propositions. This may be extended for other inference rules also. Here we have enlisted truth values of Modus Ponens and Modus Tollens rule for  propositions having linguistic truth values  that may be represented by two types of lattices (Fig.1 and Fig.2). The same method may be extended for other types of lattices where the pattern of non comparable elements are different.\\\\
\textbf{Acknowledgement}\\ We are indebted to Prof. M.K.Chakraborty  for his valuable suggestions and inspiration in the preparation of this paper.

\Addresses
\end{document}